\documentclass[a4paper,10pt,reqno]{amsart}

\usepackage[utf8]{inputenc}
\usepackage[
theoremdefs,
final
]{latexdev}
\usepackage{caption}
\usepackage{subcaption}

\usepackage[numbers]{natbib}
\usepackage{standalone}
\usepackage{tikz,tikz-cd}

\numberwithin{equation}{section}
\numberwithin{theorem}{section}

\DeclareMathOperator{\tr}{tr}

\newcommand{\pair}[1]{\left\langle #1 \right\rangle}

\newcommand{\inner}[1]{\langle\!\langle #1 \rangle\!\rangle}
\providecommand{\norm}[1]{\lVert#1\rVert}

\newcommand{\ud}{\mathrm{d}}

\newcommand{\uud}{\,\ud}

\newcommand{\RR}{{\mathbb R}}

\newcommand{\ZZ}{{\mathbb Z}}

\newcommand{\met}{\mathsf{g}}
\newcommand{\Diff}{\mathrm{Diff}}

\newcommand{\Orb}[1]{\mathrm{Orb}(#1)}

\newcommand*\SO{\mathrm{SO}}

\newcommand{\Ver}{\mathrm{Ver}}
\newcommand{\Hor}{\mathrm{Hor}}

\title{The Toda flow as a porous medium equation}
\author{Boris Khesin and Klas Modin}
\address{B.K.: Department of Mathematics,
University of Toronto, ON M5S 2E4, Canada}
\email{khesin@math.toronto.edu}
\address{K.M.: Department of Mathematical Sciences, Chalmers University of Technology and University of Gothenburg, SE-412 96 Gothenburg, Sweden} \email{klas.modin@chalmers.se}
\date{}                                           

\begin{document}

\begin{abstract}
We describe the geometry of the incompressible porous medium (IPM) equation: we prove that it is a gradient dynamical system on the group of area-preserving diffeomorphisms and has a special double-bracket form. Furthermore, we show its similarities and differences with the dispersionless Toda system. The Toda flow describes an integrable interaction of several particles on a line with an exponential potential between neighbours, while its continuous version is an integrable PDE, whose physical meaning was obscure. Here we show that this continuous Toda flow can be naturally regarded as a special IPM equation, while the key double-bracket property of Toda is shared by all equations of the IPM type, thus manifesting their gradient  and non-autonomous Hamiltonian origin. Finally, we comment on Toda and IPM modifications of the QR diagonalization algorithm, as well as describe double-bracket flows in an invariant setting of general Lie groups with arbitrary inertia operators. 
\smallskip

\textbf{Keywords:} Toda flow, porous medium equation, volume preserving diffeomorphisms, double bracket flow
\smallskip

\textbf{MSC 2020:} 37K25, 70H06, 76S05
\end{abstract}

\maketitle


\section{Introduction\label{sec:intro}}

The goal of this paper is three fold.

We start by studying the geometry of the 2D incompressible porous medium (IPM) equation and describe its relation to other equations of mathematical physics. This equation  was recently found to possess intriguing properties for a general potential and to be related to problems of optimal mass transport, see \cite{Ot1999,Br2022}. Here we show that the IPM equation is a gradient dynamical system on the Lie group of area-preserving diffeomorphisms, and moreover, its gradient nature turns out to be due to a novel double bracket form of the IPM equation. Recall that the double brackets were introduced by Brockett for matrix equations \cite{Bl1990,Bro}. Such equations manifest double nature: they turn out  to be gradient on adjoint orbits of compact Lie groups, as well as Hamiltonian with possibly time-dependent Hamiltonian functions. The  Hamiltonian viewpoint requires a certain identification of a Lie algebra and its dual, which is natural in the compact case due to bi-invariant inner product. In Appendix~A we extend the setting of double brackets to the duals of arbitrary Lie algebras by exhibiting the role of the inertia operator for the identification of a Lie algebra in its dual. It turns out that the 2D IPM equation is a double bracket equation on the dual of the Lie algebra $\mathfrak{X}_\mu(M)$ of divergence-free vector fields on $M$. In the vorticity representation, which we employ, the inertia operator is given by the Laplacian $\Delta$ on $M$. Furthermore, extending this result, we prove the local in time existence of solutions for the IPM equation on a symplectic manifold in any dimension.

The second goal of the paper is a presentation of the dispersionless Toda flow as a certain porous medium equation. The classical Toda system on a real line describes an evolution of a finite collection of particles with neighbours interacting with exponential potential. The corresponding continuous limit, the dispersionless Toda flow (see e.g.\ \cite{BrBl1990, Bloch}),  can be thought of as a flow of continuum of particles mutually interacting with local exponential-type potential. We show that this interaction can be regarded as a porous medium interaction with a different (``the identity") inertia operator on a certain subspace of the dual $\mathfrak{X}_\mu^*(M)$. This way the dispersionless Toda dynamics turns out to be a special type of percolation of faster particles through slower ones.

Finally, this new IPM point of view on the Toda flow suggests an unusual modification of the classical QR diagonalization algorithm for symmetric matrices, particularly efficient for matrices of large sizes. It was shown in \cite{Deift1} that the QR diagonalization algorithm is a discrete time version of the Toda flow. It turns out that while theoretically convergent, its numerics is rather unstable and its speed of convergence to reordered eigenvalues is rather slow. The IPM equation with the Laplace inertia operator turns out to be an efficient replacement for that, as thanks to the presence of $\Delta^{-1}$ in the equation, its flow is stable and very fast converging. We demonstrate this in Section \ref{sec:comments} by presenting diagonalization process of a $256\times 256$ matrix by two algorithms, where, even after 1000 iterations, the Toda flow keeps oscillating, while the IPM flow obtains and reorders the eigenvalues much faster. We also propose yet another modification of this algorithm, 
providing a fast diagonalization of large symmetric matrices without ordering their eigenvalues.

\bigskip

\textbf{Acknowledgements.} 
We thank the anonymous referee for helpful suggestions.
Research of B.K.\ was partially supported by an NSERC Discovery Grant.
K.M.\ was supported by the Swedish Research Council (grant number 2022-03453) and the Knut and Alice Wallenberg Foundation (grant number WAF2019.0201).
The computations were enabled by resources provided by the Swedish National Infrastructure for Computing (SNIC) at C3SE partially funded by the Swedish Research Council through grant agreement no.~2018-05973.


\section{The IPM equation}\label{sec:ipm}

\subsection{Definition of the IPM flow}

\begin{definition}
Given a potential function $V$ on a manifold $M$ of an arbitrary dimension, the \emph{incompressible porous medium} (IPM or Muscat) equation is given by
\begin{equation}\label{eq:ipm}
    \dot\rho + \operatorname{div}(\rho v) = 0, 
\end{equation}
subject to the conditions   
$$  v + \rho \nabla V = - \nabla p ,
    \quad \operatorname{div} v = 0.
$$
This is a continuity equation on the density $\rho$ transported by the divergence-free part $v$ of the vector field $\rho\nabla V$ on $M$, see \cite{Ot1999,Br2022}.
The dot stands for the derivative in time.
Note that the pressure term is determined by the condition
\begin{equation}
    \operatorname{div}(\rho \nabla V) = -\Delta p\,.
\end{equation}

The IPM equation is also called  a Muscat equation, where MUSCAT stands for Multiple-Doppler Synthesis and Continuity Adjustment Technique. 
\end{definition}

If $M$ is a K\"ahler 2-manifold, the setting we are going to consider from now on, 
the field $v$ has a (locally defined) stream function $\psi$, $v:=-J\nabla \psi$. 
Assuming that the stream function is defined globally on $M$ 
we get
\[
    \nabla \psi = \rho J \nabla V + J\nabla p
    \iff
    \Delta \psi = \operatorname{div}(\rho J\nabla V) = \nabla\rho\cdot J\nabla V
    = \{\rho, V \}\,.
\]
We thereby obtain a ``vorticity formulation'' of the IPM equation as
\begin{equation}\label{eq:vort_ipm}
    \dot\rho = \{\psi,\rho \}, \quad \Delta\psi = \{V,\rho\}.
\end{equation}

\begin{remark}
The first equation emphasizes the Hamiltonian nature of the evolution of density $\rho$ with the (time-dependent) Hamiltonian $\psi$
with respect to the area form on $M$.
In particular, $\rho$ preserves all the standard Casimir functions
\[
    C_f(\rho) = \int_M f(\rho(x))d x
\]
for any (measurable) functions $f\colon \mathbb{R}\to\mathbb{R}$.
\end{remark}


\subsection{Gradient flows on volume preserving diffeomorphisms}
It turns out, the IPM equation is not only Hamiltonian, but it has a more prominent gradient nature. Namely, let $M$ be a two-dimensional Riemannian manifold
and $\operatorname{Diff}_\mu(M)$ stands for the group of Hamiltonian diffeomorphisms.

We are interested in gradient flows on $\operatorname{Diff}_\mu(M)$ with respect to a right-invariant Riemannian metric defined by an inertia operator $A\colon \mathfrak{X}_\mu(M) \to \mathfrak{X}_\mu^*(M)$.  For the group $\operatorname{Diff}_\mu(M)$ we may identify the dual $\mathfrak{X}_\mu^*(M)$ of its Lie algebra with the space of exact 2-forms $d\Omega^1(M)$ or with the space of  vorticity functions $C^\infty_0(M)$.
The pairing for a Hamiltonian vector field $v = X_\psi$ and a vorticity $\omega\in C^\infty_0(M)$ is
\begin{equation}
	\langle v, \omega\rangle = \int_M \psi\,\omega\, \mu .
\end{equation}
With respect to this pairing the $\operatorname{ad}^*$ operator is given by the bracket
\begin{equation}
	\operatorname{ad}^*_v\omega = \{\psi,\omega \}.
\end{equation}
Since the Lie algebra of Hamiltonian vector fields can be identified with the Poisson algebra of Hamiltonian functions (modulo constants), 
from now on we shall implicitly identify $v$ with its Hamiltonian (or stream) function $\psi$ via $v = X_\psi$.

Let $V\in C^\infty(M)$ be a potential function for a classical mechanical system with configuration space $M$.
Now we fix a function  $\omega_0\in C^\infty(M)$, which will have the meaning of density or vorticity depending on the setting.
Consider now a ``linear" functional on the group $\operatorname{Diff}_\mu(M)$ associated  with this fixed function of the form
\begin{equation}\label{eq:energy}
	E_0(\varphi)  = \int_M (V\circ\varphi) \, \omega_0\,\mu = \int_M V \, \omega \,\mu ,
\end{equation}
where $\omega :=\omega_0\circ\varphi^{-1}$ is the  vorticity
transported by a diffeomorphism $\varphi\in \operatorname{Diff}_\mu(M)$.

\begin{theorem}
	The gradient flow on $\operatorname{Diff}_\mu(M)$ for the functional \eqref{eq:energy} and the right-invariant Riemannian metric determined by the inertial operator $A$ is
	\begin{equation}
		\dot\varphi = -\nabla E_0(\varphi) = - X_\psi\circ\varphi, \quad \text{where} \quad A\psi = \{V,\omega_0\circ\varphi^{-1}\}.
	\end{equation}
	The dynamics is completely determined by the transported vorticity $\omega = \omega_0\circ\varphi^{-1}$ via the ``double bracket''-type  flow
	\begin{equation}\label{eq:gen_double_bracket}
		\dot\omega = -\{A^{-1}\{V, \omega \}, \omega \} 
	\end{equation}
	for the inertial operator $A=\Delta$ given by the Laplacian, $A: \psi\mapsto \omega=\Delta\psi$.
\end{theorem}

\begin{proof}
This is an application of the general framework described in Appendix A (Section~\ref{sec:universal}). By setting
$E_0(\varphi) = F(\omega_0\circ\varphi^{-1})$ for $F(\omega) = \int_M V\omega \,\mu$, 
the gradient flow for $E_0$ corresponds to the dynamics of the generator $\omega$ given above. Namely, it is a double bracket flow  with the inertia operator squeezed in between the two brackets.
\end{proof}

\begin{example} Take the manifold $M$ to be the cylinder, $M = T^*S^1=\{(\phi, z)~|~\phi\in S^1, \, z\in \RR\}$, while the potential $V(\phi,z) = z$ to be the height function.

$a)$ For the inertia operator $A=\Delta$ on stream functions we obtain the incompressible porous medium equation (IPM) defined above. 
Note that $A$ defines the standard $L^2$ Arnold metric, when considered on  vector fields (instead of their stream functions), and it is 
the $H^{-1}$-metric when considered on vorticities. The potential functional is $ F(\rho) = \int_M V\rho$.


$b)$ For the identity inertia, $A=id$, and the same potential we obtain the continuous Toda flow defined in the next section.
\end{example}

\begin{remark}
Notice that the functional $F(\rho) = \int_M V\rho$ can be also regarded as a Lyapunov function for the IPM flow with respect to the $H^{-1}$-metric:
\[
    \frac{d}{dt} F(\rho) = \int_M V \{\psi,\rho \}
    = \int_M \{V,\rho\}\Delta^{-1}\{V,\rho\}
    = -\langle \{V,\rho\},\{V,\rho\}\rangle_{H^{-1}}.
\]
\end{remark}

\begin{remark}
It is worth to add that the IPM equation (as well as the Toda system discussed below), being a ``double bracket" flow, exhibit both gradient and Hamiltonian properties: it is a gradient  flow for a function on a coadjoint orbit, but it is also a Hamiltonian flow for another function on the same orbit. Note, however, that  this is an example of a {\it non-autonomous} Hamiltonian equation: the Hamiltonian function does change in time, as its differential is related to a certain bracket itself. Thus the convergent asymptotic behavior typical for gradient systems (and discussed below as a diagonalization procedure) does not contradict the conservative nature of {\it autonomous} Hamiltonian systems.
\end{remark}



\subsection{Local well-posedness of IPM}

In this section we use the techniques developed by \citet{EbMa1970} to prove local in time well-posedness of the IPM equation \eqref{eq:vort_ipm}. 
Short-time existence results for $M=\mathbb{R}^2$ were already given by \citet*{CoGaOr2007}.
The alternative approach described below is of interest due to a much closer relation to Euler-Arnold equations and previous geometric analysis framework.

\begin{theorem}\label{thm:local_wellposed}
	Let $M$ be a closed Kähler (Riemannian and symplectic) manifold of dimension $2d$ and let $V\in C^\infty(M)$.
	Furthermore, let $s>d+1$.
	Then for every initial $\rho_0\in H^s(M)$ there exists a maximal existence interval $[0,T)$ with $T>0$ in which the equation \eqref{eq:vort_ipm} has a unique solution $\rho\colon [0,T)\to H^s(M)$.
\end{theorem}

\begin{proof}
	First, we relax the equation by lifting it to the group of symplectomorphisms $\operatorname{Diff}_\mu(M)$:
	\begin{equation}\label{eq:ipm_diffeoform}
		\dot\eta = v\circ \eta, \quad v = \nabla^\perp \Delta^{-1}\{V,\rho \}, \quad \eta(0) = \operatorname{id},
	\end{equation}
	where $\rho = \rho_0\circ\eta^{-1}$.
	Since $s>d+1$ it follows from results by \citet{Pa1968} and by \citet{EbMa1970} that the Sobolev completion $\operatorname{Diff}^s_\mu(M)$ is a Hilbert manifold.
	The aim is to show that the equation \eqref{eq:ipm_diffeoform} is a smooth ordinary differential equation on $\operatorname{Diff}^s_\mu(M)$.
	We have that $v$ can be thought of as the following composition: $v = A\circ B\circ C \circ\tau (\eta^{-1})$ where
	\begin{align}
		A\colon & H^{s+1}_0(M)\to \mathfrak{X}_{\mu}^s,\quad \psi \mapsto \nabla^\perp \psi \\
		B\colon & H^{s-1}_0(M)\to H^{s+1}_0(M),\quad \rho \mapsto \Delta^{-1}\rho \\
		C\colon & H^{s}(M)\to H^{s-1}_0(M),\quad \rho \mapsto \{V,\rho \} \\
		\tau\colon & \operatorname{Diff}_\mu(M) \to H^s(M),\quad \varphi \mapsto \rho_0\circ\varphi.
	\end{align}
	The flow can now be written as
	\begin{equation}
		\dot\eta = \big((A\circ B\circ C\circ \tau)(\eta^{-1})\big)\circ \eta .
	\end{equation}
	This proves that \eqref{eq:ipm_diffeoform} is indeed an ODE on $\Diff^s_\mu(M)$.
	However, it remains to prove that it is a \emph{smooth} ODE.

	The operators $A,B,C,\tau$ are all smooth, but the operation $\eta\to\eta^{-1}$ is not more than $C^0$ (in particular not Lipschitz).
	The strategy of \citet{EbMa1970} is to prove that the particular combination of the `inverse map--differential operator--forward map' is smooth.
	To this end, we define $\Diff^s_\mu(M)$-bundle operators
	\begin{align}
		\tilde A \colon & \Diff^s_\mu(M)\times H_0^{s+1}\to T\Diff_\mu(M), \quad (\eta,\psi) \mapsto (A(\psi\circ\eta^{-1}))\circ\eta, \\
		\tilde B\colon & \Diff^s_\mu(M)\times H_0^{s-1} \to \Diff_\mu(M)\times H_0^{s+1}, \quad (\eta,\rho) \mapsto (B(\rho\circ\eta^{-1}))\circ\eta \\
		\tilde C\colon & \Diff^s_\mu(M)\times H^{s} \to \Diff_\mu(M)\times H_0^{s-1}, \quad (\eta,\rho) \mapsto (C(\rho\circ\eta^{-1}))\circ\eta .
	\end{align}
	The total flow can now be written as
	\begin{equation}
		\dot \eta = (\tilde A\circ\tilde B\circ\tilde C)(\eta,\rho_0).
	\end{equation}
	Since $A$ and $C$ are differential operators, and since the right composition is smooth, it follows from standard results (e.g.\ \cite[Lemma~3.4]{Mo2015}) that $\tilde A$ and $\tilde C$ are smooth bundle maps.
	By the same argument it follows that $\tilde B^{-1}$ is a smooth bundle map. 
	Since $B$ is an isomorphism it  follows that $\tilde B$ is also a smooth bundle map (e.g.\ \cite[Lemma~3.2]{Mo2015}).
	Then the standard Picard iterations on Banach manifolds  give local well-posedness in the sense of Hadamard.
\end{proof}

\begin{remark}
	In the next section we shall see that the Toda flow is geometrically exactly of the same form as IPM.
	The only difference is the choice of inertia operator $A$, which is the Laplacian for IPM but the identity for Toda.
	From an analysis point of view, however, the IPM flow behaves much better than Toda.
	In particular, the continuous Toda flow (see below) cannot be rigorously posed as an ODE on  the Hilbert manifold $\Diff_\mu^s(M)$: it is prevented by a loss of spatial derivatives in the right-hand side.
	As we discuss and illustrate in \autoref{sec:comments}, this lack of regularity suggests slower convergence to a steady state.
\end{remark}

\section{Toda flows, their continuous limit, and double bracket representation}

\subsection{Toda lattice}

Consider $n$ interacting particles on a line, where neighbours  interact with an exponential potential.
The corresponding system is described by 
the \textit{Toda flow} in the symplectic space $(\RR^{2n}, \sum dq_i\wedge dp_i)$ generated by the Hamiltonian $H=\frac{1}{2}\sum_{j=1}^n p_j^2+\sum_{j=1}^{n-1}\exp 2(q_j-q_{j+1})$, being the sum of kinetic and potential energies, with the interaction potential $\exp(u)$ where $u:=2(q_k-q_{k+1})$ stands for the doubled distance between neighbours. (One usually normalizes the system so that the mass center is at $0$, i.e., $\sum q_i=0$, so there are $n-1$ degrees of freedom.) 

\begin{proposition}[\citet{Fl1974}, \citet{Mo1975}] 
Newton's equation of the Toda flow of $n$ interacting particles after the change of variables $a_j=\exp (q_j-q_{j+1})$ with $j=1,\dots, n-1$, and $b_k=p_k$ with $k=1,\dots,n$, assumes the following form: 
\begin{equation}\label{eq:toda}
\begin{cases}
	&\dot a_k=-a_k(b_{k+1}-b_k)\\
	&\dot b_k=-2(a_k^2-a_{k-1}^2)\,.
\end{cases}
\end{equation}

	This system has the Lax form $\dot{L}=[L,M]$, where
		\begin{align*}
		L=\begin{pmatrix}
		b_1 & a_1& 0 &  & 0\\
		a_1 & b_2 & a_2&  & \\
		0& a_2 & b_3& & \\
		& & & \ddots &a_{n-1}\\
		0& & & a_{n-1} & b_n
		\end{pmatrix} ~\text{ and } ~
		M=\begin{pmatrix}
		0 & a_1& 0 &  & 0\\
		-a_1 & 0 & a_2&  & \\
		0& -a_2 & 0& & \\
		& & & \ddots &a_{n-1}\\
		0& & & -a_{n-1} & 0
		\end{pmatrix},
	\end{align*}
	with $\sum b_k=0$. 
		This system  is Hamiltonian with the Toda Poisson structure given by $\{b_j,a_{j-1}\}'=-a_{j-1}$, $\{b_j,a_j\}'=a_j$, while all other brackets are zero. 
		The Hamiltonian function in the new coordinates is $H=\frac{1}{2}\tr(L^2)= \sum^{n-1}_1 a_j^2 +\sum^n_1 b_k^2$, with evolution equations being 
		$$
		\dot a_j=\{H,a_j\}', \quad \dot b_j=\{H,b_j\}'.
		$$ 	
\end{proposition}

\begin{remark}
	The Lax equation $\dot{L}=[L,M]$ implies that $L$ changes in its conjugacy class, $L(t)=u(t)L_0u^{-1}(t)$, where $u(t)\in \SO(n)$ are solutions to $\dot{u}=Mu$, $u(0)=I$. In particular, functions $I_k=\frac{1}{k}\tr(L^k)$, that are symmetric polynomials of the eigenvalues of $L$,  are first integrals for $k=2,\dots, n$. In particular, $I_2=H$. 	
(In Lie algebra language, $\dot{L}=[L,M]$ is an equation on a symplectic manifold, which is the orbit of $L_0$ understood as an orbit in the appropriate 
lower-trangular matrix group, while $\omega$ is the orbit symplectic structure coming from the Lie-Poisson bracket in that space.)
The existence of $n-1$ first integrals, which turn out to be in involution, on the symplectic $2(n-1)$-dimensional  space  implies complete integrability of the Toda flow.
\end{remark}

\begin{remark}
The matrix $M$ itself can be represented as the commutator: $M=[L, D]$, where $D\coloneqq \mathrm{diag}(1,2,...,n)$.
This implies that the Toda system of $n$ particles can be rewritten as the double-bracket equation:
\begin{equation}\label{eq:toda_double_bracket_form}
	\dot L=[L,[L,D]]	
\end{equation}
for the matrix $L$ as above, see \citet*{Bl1990} and \citet*{BlBrRa1992}.
\end{remark}

\medskip


\subsection{Toda continuous limit}
Following \citet*{Bloch}, consider now the continuous limit of the above system. 

\begin{definition}
The continuous Toda flow (or dispersionless Toda system of equations) is the following evolution system on functions $a(z), \, b(z)$ of PDEs:
\[
\begin{cases}
	&\dot a =  -a \frac{\partial}{\partial z}b\\
	&\dot b =  -2\frac{\partial}{\partial z}a^2\,.
\end{cases}
\]
\end{definition}

There are several ways to obtain this system from its discrete analogue. One can directly replace the discrete parameter $k\in \ZZ$ in the Toda system \eqref{eq:toda}
by a continuous parameter $z\in \RR$, while replacing the differences of consecutive terms by the derivatives.

Alternatively (see \cite{Bloch}),  one can think of $L$ and $M$ as infinite tridiagonal matrices and express them in the form
$L=a_k \exp{(\partial/\partial \phi)} +b_k + a_k \exp{(-\partial/\partial \phi)}$ and 
$M=a_k \exp{(\partial/\partial \phi)} -  a_k \exp{(-\partial/\partial \phi)}$, where 
$\exp{(\partial/\partial \phi)} $ is understood as the diagonal shift operator.  (Indeed, if $\phi$ enumerates diagonals and
then according to the Taylor expansion, $\exp{(\epsilon\partial/\partial \phi)} $ is a shift by $\epsilon$: for a test function $f(\phi)$ one has
$f(\phi+\epsilon)=\sum_m \epsilon^m f^{(m)}(\phi)/m! = \exp(\epsilon\partial/\partial \phi)f(\phi)$ as $\epsilon\to 0$.)

Now  replace the integer index $k$ of $a_k$ and $b_k$ by a continuous parameter $z\in \RR$ to obtain time-dependent functions $a(z),\, b(z)$.
The corresponding system of equations becomes the system of \emph{dispersionless Toda equations}.
The continuous limit of the tridiagonal matrix $L$ assumes the form $ L(z,\phi) = b(z) + 2a(z) \cos\phi$, where  $2\cos\phi=\exp(  i \phi)+\exp(- i \phi)$ and exponentials $\exp(\pm  i \phi)$  label the first super- and sub-diagonals, while variable $z$ parametrizes the diagonal direction. (The above form of $L$ can be thought of as a compact version of a hyperbolic analog $L= b(z) + 2a(z) \cosh\phi$, where  $2\cosh\phi=\exp \phi +\exp(-\phi)
$.)
\begin{proposition}[Faybusovich 1990, see \cite{Bloch}]
For the function 
$$ 
L(\phi,z) = b(z) + 2a(z) \cos\phi
$$ on $M=T^*S^1$ the dispersionless Toda equations assume the form
 of the Brockett double bracket equation
\begin{equation}\label{eq:mainL}
\frac{dL}{dt}=-\{L,\{L,z\}\}\,,
\end{equation}
where $\{\cdot,\cdot\}$ is the Poisson bracket for the symplectic structure $dz\wedge d\phi$
on the annulus $M=\{(z,\phi)~|~z\in \RR, \phi \in S^1\}$.
\end{proposition}

\begin{proof}
It is a straightforward computation:
    \begin{equation}
        \{L,z\} = \frac{\partial L}{\partial \phi}\frac{\partial z}{\partial z} - \frac{\partial z}{\partial \phi}\frac{\partial L}{\partial z} = - 2 a(z) \sin\phi
    \end{equation}
    Then
$$      \{L, \{L,z\}\} = \{L, - 2 a(z) \sin\phi \} 
  $$
  $$
       = \frac{\partial L}{\partial \phi}(- 2 a'(z) \sin(\phi)) +  2 a(z) \cos\phi(b'(z) + 2 a'(z)\cos\phi)
  $$ 
$$
= 2 a(z) \sin\phi(2 a'(z) \sin\phi) + 2 a(z) \cos\phi(b'(x) + 2 a'(z)\cos\phi)
$$ 
    \begin{equation}
        =4 a(z)a'(z)  + 2 a(z) b'(z) \cos\phi
    \end{equation}
    Thus the equation $\frac{dL}{dt}=-\{L,\{L,z\}\}$ for $ \dot L:=    \dot b(z) + 2\dot a(z)\cos\phi$ is equivalent to the system
 \begin{equation}\label{eq:abflowab}
 \begin{cases}
	& \dot a =-a  b' \\
	&\dot b =-2(a^2 )'\,,
\end{cases}
\end{equation}
    where prime stands for the derivative in $z$.
\end{proof}

\medskip

\begin{remark}
In the continuous limit the integrals  $I_k = \operatorname{tr}(L^k), \, k=1,2,...$ become 
$$
J_k=\frac{1}{2\pi}\int_\RR \int_0^{2\pi} (b(z)+2a(z)\cos \phi)^k\,dz\,d\phi\,.
$$
In particular, $J_2= \int_\RR (b(z)^2+2 a(z)^2)\,dz$.

\smallskip

Yet one more way of obtaining the continuous Toda flow directly from the coordinates in the $(q,p)$ particle phase space
is given in  Appendix~B.
\end{remark}

\medskip


\subsection{The Hamiltonian nature of the Toda flow}
Consider the Poisson structure
\begin{equation}
	\{F,G\}'(a,b) = \int_\mathbb{R} a \frac{\delta F}{\delta a}\frac{\partial}{\partial x}\frac{\delta G}{\delta b} - a \frac{\delta G}{\delta a}\frac{\partial}{\partial x}\frac{\delta F}{\delta b} .
\end{equation}
Then Hamilton's equations become
\begin{align}
	\dot a &= \{H,a \}' = - a \frac{\partial}{\partial x}\frac{\delta H}{\delta b} \\
	\dot b &= \{H,b \}' = - \frac{\partial}{\partial x} a\frac{\delta H}{\delta a} .
\end{align}
Thus we see that \eqref{eq:abflowab} is Hamiltonian for
\begin{equation}
	H(a,b) = \frac{1}{2}\int_\mathbb{R} (2a^2 + b^2) \, \ud z = J_2(a,b)/2.
\end{equation}
Notice that $\operatorname{ad}^*$ is given by
\begin{equation}
	\operatorname{ad}^*_{(u,v)}(a,b) = \begin{pmatrix} -a\frac{\partial}{\partial x}v \\ -\frac{\partial}{\partial x}a u \end{pmatrix}
\end{equation}
where $(u,v)$ are algebra elements. (Here we assume that $a^2 b$ vanishes at $\pm\infty$ so the boundary terms vanish when we do integration by parts.)
Thus, we have shown that the system \eqref{eq:abflowab} admits both a Hamiltonian and a gradient formulation.


\medskip


\subsection{A comparison of the continuous Toda and the IPM equation}
Now we return to the double-bracket form of both the equations. 
Recall that the IPM equation on a 2D manifold  can be thought of as the following equation 
on the vorticity $\omega$, cf.\ equation \eqref{eq:gen_double_bracket}:
$$
\dot\omega = -\{A^{-1}\{z, \omega \}, \omega \} .
$$
On the other hand, the dispersionless Toda equation \eqref{eq:mainL}
has the form 
$$
\dot L=-\{\{z,L\}, L\}
$$
on the function $L(\phi,z) = b(z) + 2a(z) \cos\phi$, where  $b$ plays the role of the (continuous) momentum and $a$ would be 
related to the particle density on the line, cf.\ equation \eqref{eq:notations}.

\begin{remark}
For the interpretation as a porous medium equation  the following feature of the Toda dynamics of $n$ particles will be important:
$a_k(t)\to 0$ and $b_k (t)\to \text{const}_k$ as $t\to \pm \infty$ for $k=1,\dots, n-1$. This implies that the 
distance between particles goes to infinity, $|q_{k+1}(t)-q_k(t)|\to \infty$, and $L(t)\to \operatorname{diag}(p_1^\pm,\dots,p_n^\pm)$. Furthermore, due to the isospectral property, $p_k^\pm$ are  eigenvalues, that are assumed to be distinct, while the interaction is repelling, with $p_j(t)\to p_j^\pm$ and $q_j(t)\to p_j^\pm t+q_j^\pm$ as $t\to \infty$. If $\lambda_1<\dots<\lambda_n$ are ordered eigenvalues of $L(t)$ for all $t$, then $p_j^+=\lambda_j$, $p_j^-=\lambda_{n+1-j}$ and $\lambda_j$ are first integrals, implying integrability [Moser 1975]. 

In other words, as $t\to \pm \infty$ the particles are almost non-interacting, while as $t$ changes from $-\infty$ to $+\infty$ their interaction (by the Toda exponential potential)
can be thought of as the faster particles ``penetrating through" slower ones. This can also be interpreted as rearranging the order of particles to have their momenta increasing on the line for large  $t$.
\end{remark}

\smallskip

Now one can interpret the continuous Toda flow as a flow of particles on the line, in which particles ``go through each other" as through porous medium 
and undergo certain interactions, according to the equations. The interaction law is determined by the inertia operator $A$ and is different for the 
IPM equation and the Toda flow, where $A$ is respectively $\Delta$ or the identity $id$. (The difference in the inertia operator is somewhat similar to different shapes of rigid bodies leading to different Euler top equations according to those shapes.) Hence the Toda flow can be regarded as a certain seepage in the porous medium.

\begin{figure}
	\centering
	\begin{subfigure}[b]{0.49\textwidth}
		\centering
		\caption{Initial time}
		\includegraphics[width=\textwidth]{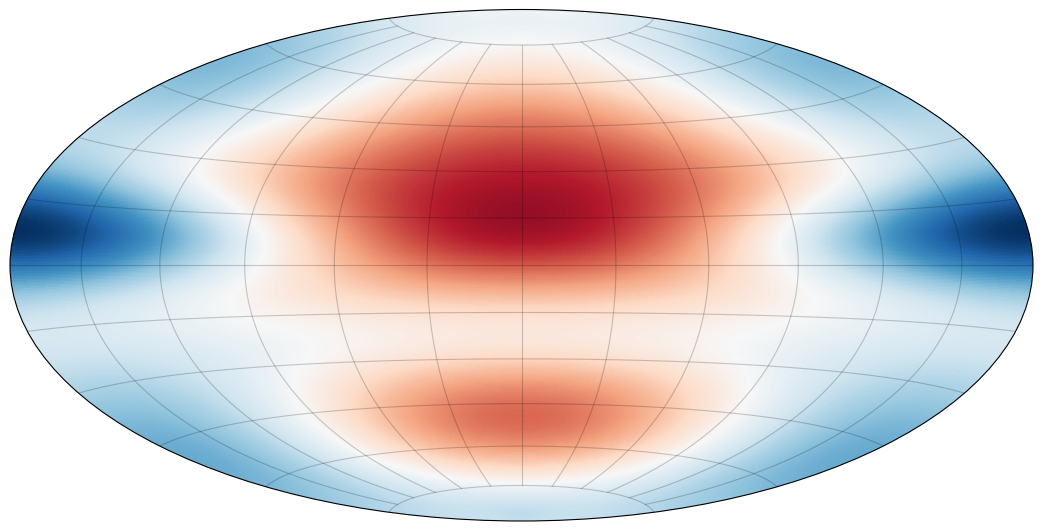}
	\end{subfigure}
	\hfill
	\begin{subfigure}[b]{0.49\textwidth}
		\centering
		\caption{First intermediate time}
		\includegraphics[width=\textwidth]{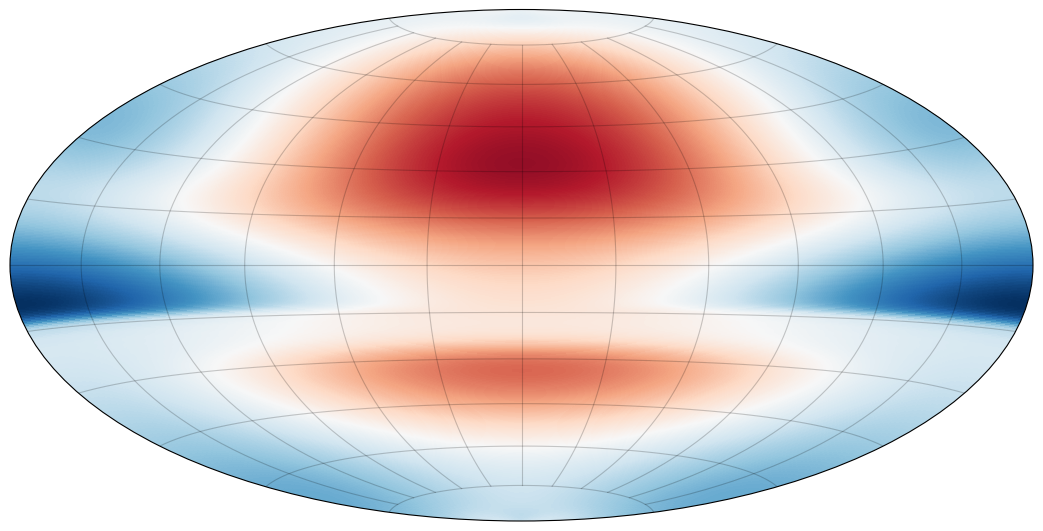}
	\end{subfigure}
	\\
	\begin{subfigure}[b]{0.49\textwidth}
		\centering
		\caption{Second intermediate time}
		\includegraphics[width=\textwidth]{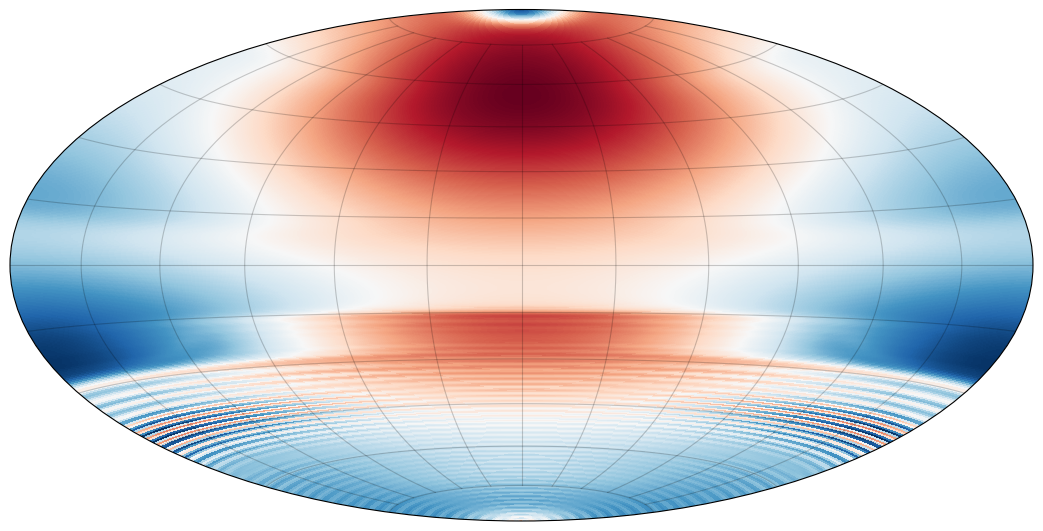}
	\end{subfigure}
	\hfill
	\begin{subfigure}[b]{0.49\textwidth}
		\centering
		\caption{Final time}
		\includegraphics[width=\textwidth]{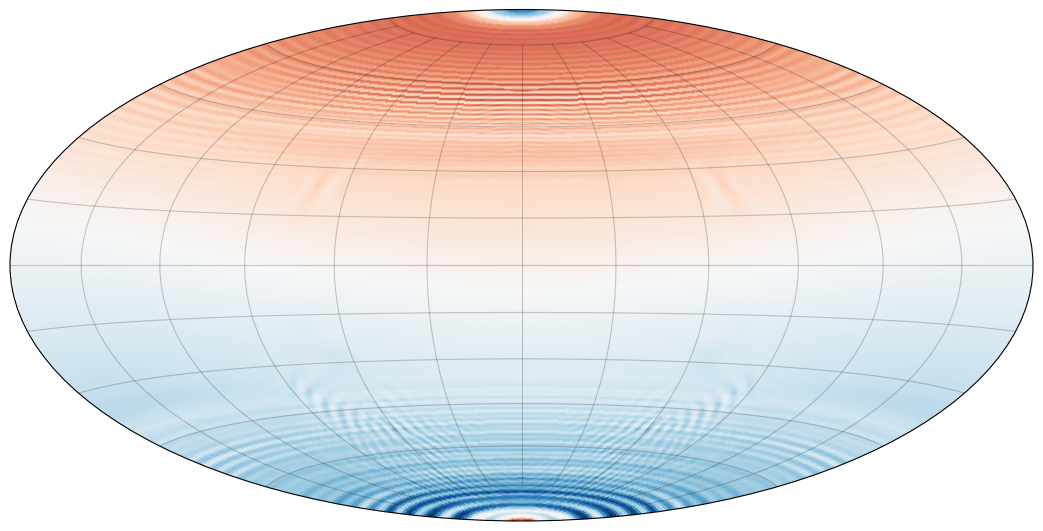}
	\end{subfigure}
	\caption{Fluid interpretation of the continuous Toda flow. As a gradient flow, it strives to move the positive parts (red) to the north and the negative parts (blue) to the south while also being zonal (corresponding to the matrix being diagonalized). }
	\label{fig:todaflow}
\end{figure}

\begin{figure}
	\centering
	\begin{subfigure}[b]{0.49\textwidth}
		\centering
		\caption{Initial time}
		\includegraphics[width=\textwidth]{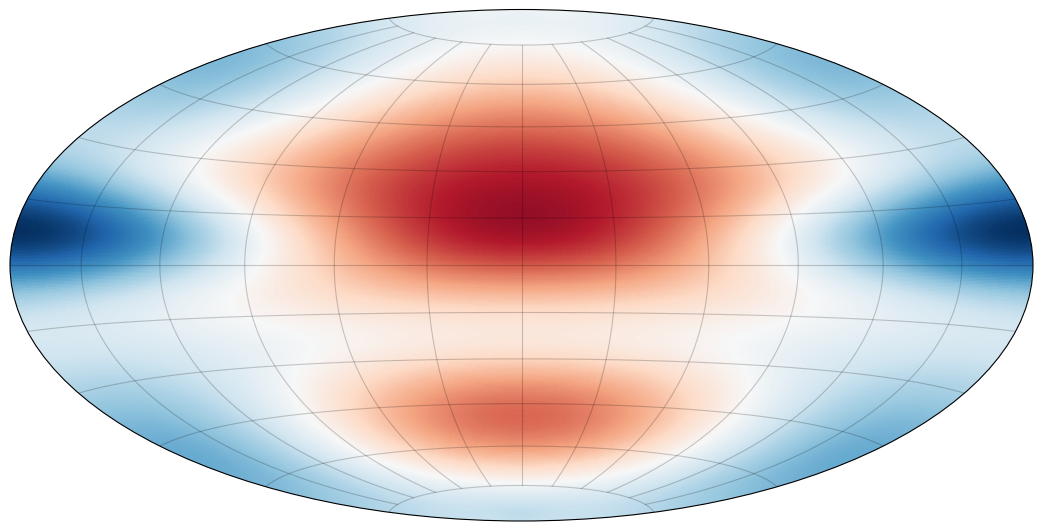}
	\end{subfigure}
	\hfill
	\begin{subfigure}[b]{0.49\textwidth}
		\centering
		\caption{First intermediate time}
		\includegraphics[width=\textwidth]{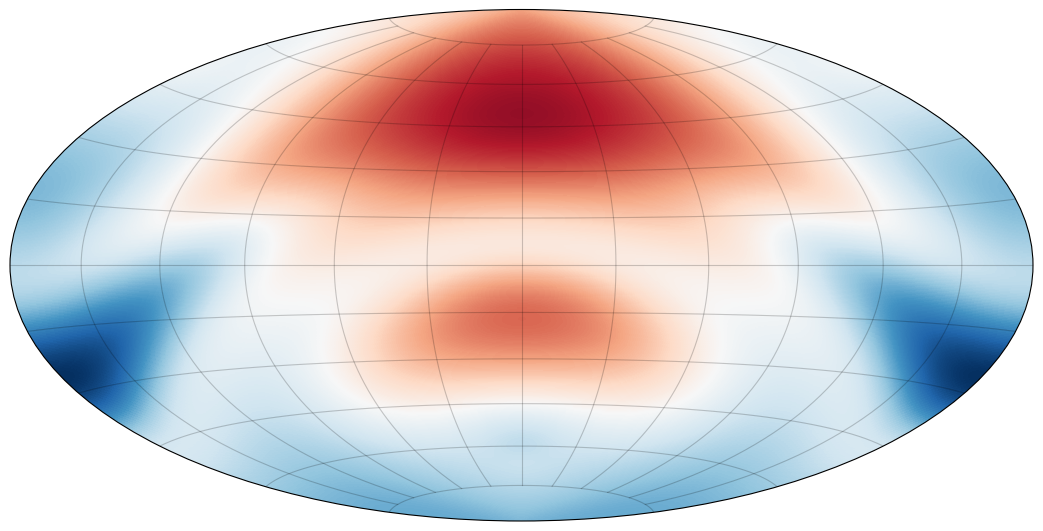}
	\end{subfigure}
	\\
	\begin{subfigure}[b]{0.49\textwidth}
		\centering
		\caption{Second intermediate time}
		\includegraphics[width=\textwidth]{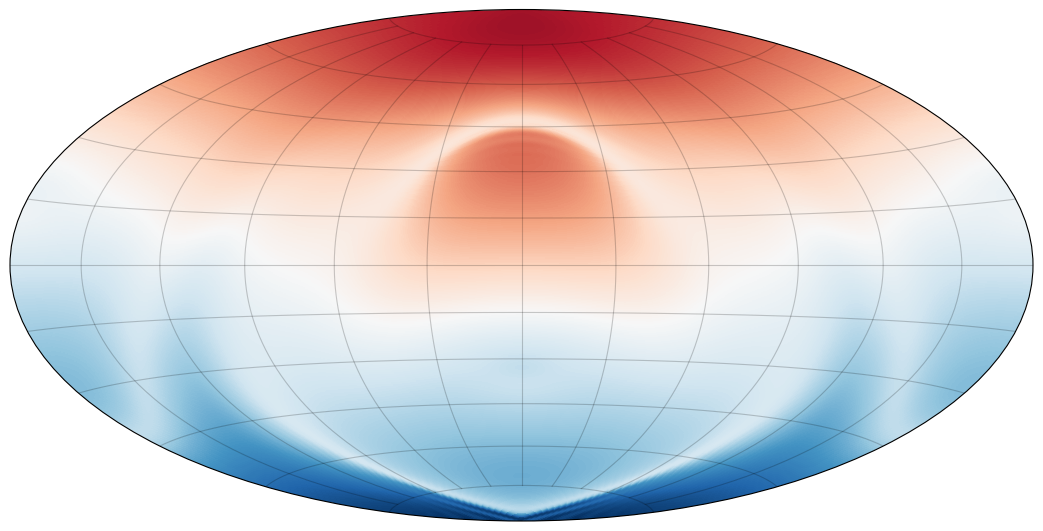}
	\end{subfigure}
	\hfill
	\begin{subfigure}[b]{0.49\textwidth}
		\centering
		\caption{Final time}
		\includegraphics[width=\textwidth]{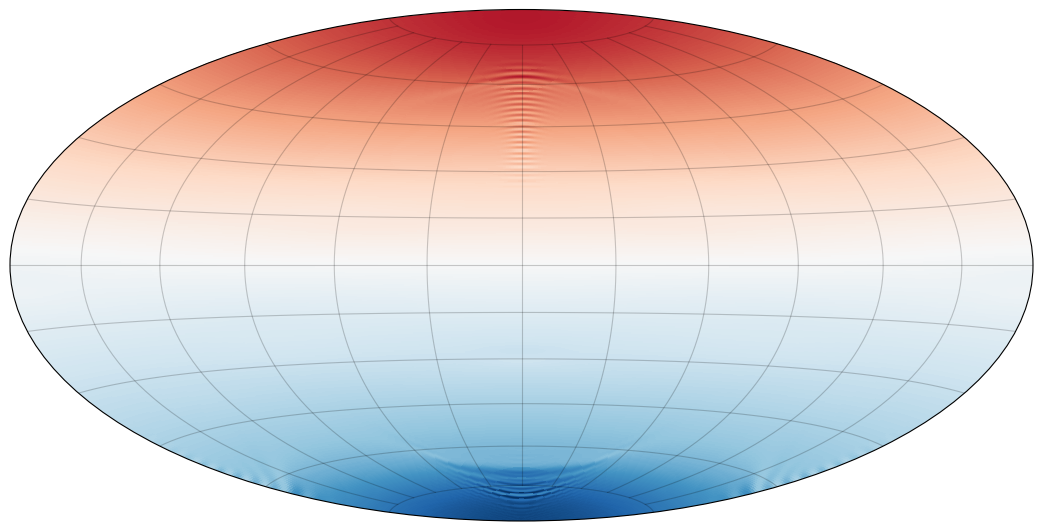}
	\end{subfigure}
	\caption{The IPM flow on the sphere. Note that being gradient and similarly to Toda, it also strives to move the positive parts (red) to the north and the negative parts (blue) to the south while also being zonal. This motion is achieved differently than for the Toda flow: it is smoother and the corresponding numerical scheme is more stable.}
	\label{fig:ipmflow}
\end{figure}

\begin{example}
We present modelling of the two flows, IPM in equation~\eqref{eq:gen_double_bracket} and continuous Toda in equation~\eqref{eq:mainL}, on the sphere $M=S^2$ to see the corresponding similarity and differences in their dynamics. 
Details of how the numerical simulations were carried out are given in section~\ref{sec:comments} below.

We note that in the case of Toda, the dynamics in $a(z)$ and $b(z)$ is visualized as a dynamics on $L$ of special tridiagonal form (or its smooth analog of the form $L(\phi,z) = b(z) + 2a(z) \cos\phi$ with auxiliary variable $\phi$). 
Such a submanifold in the space of all matrices or all functions $L(z,\phi)$ is invariant but unstable, which makes the modeling numerically challenging. 
The corresponding evolution clearly shows the penetration of particles, as they try to align with the potential, which is the height function. See \autoref{fig:todaflow} for dynamics of functions $L\colon S^2\to \mathbb{R}$ as a spherical analog of the cylinder.

On the other hand, the dynamics of the IPM equation thanks to the (inverse inertia) operator $\Delta^{-1}$ in Equation \eqref{eq:gen_double_bracket}, the corresponding dynamics of the vorticity function $\omega$ is numerically stable. 
Hence, while an analogous submanifold of functions $\omega$ is not invariant, it does not affect the asymptotic dynamics: solutions still converge to a similar final configurations aligned with the potential, see \autoref{fig:ipmflow}. These similarities in the IPM and Toda flows allow one to interpret the latter as a porous medium-type equation.
\end{example}


\begin{figure}
	\centering
	\includegraphics[width=0.95\textwidth]{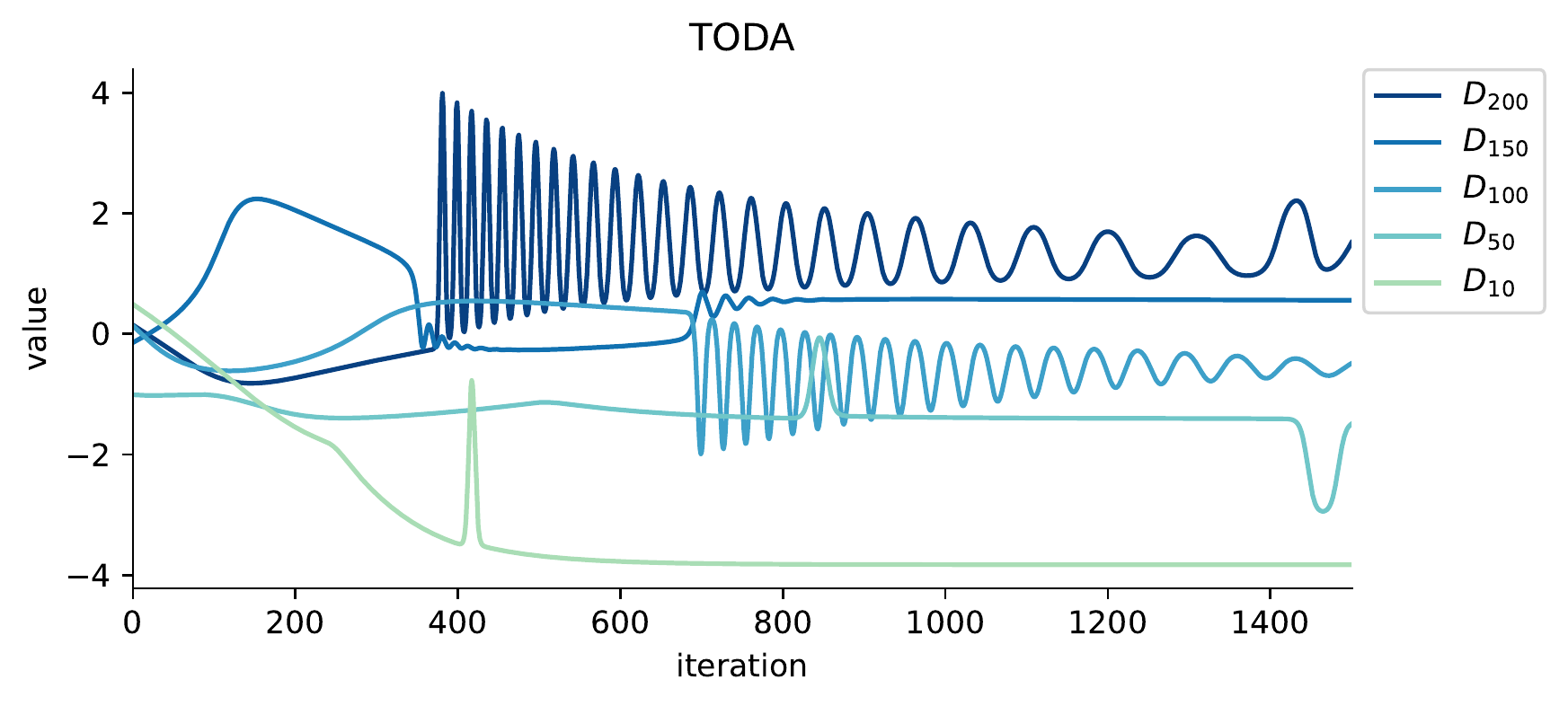}\\
	\includegraphics[width=0.95\textwidth]{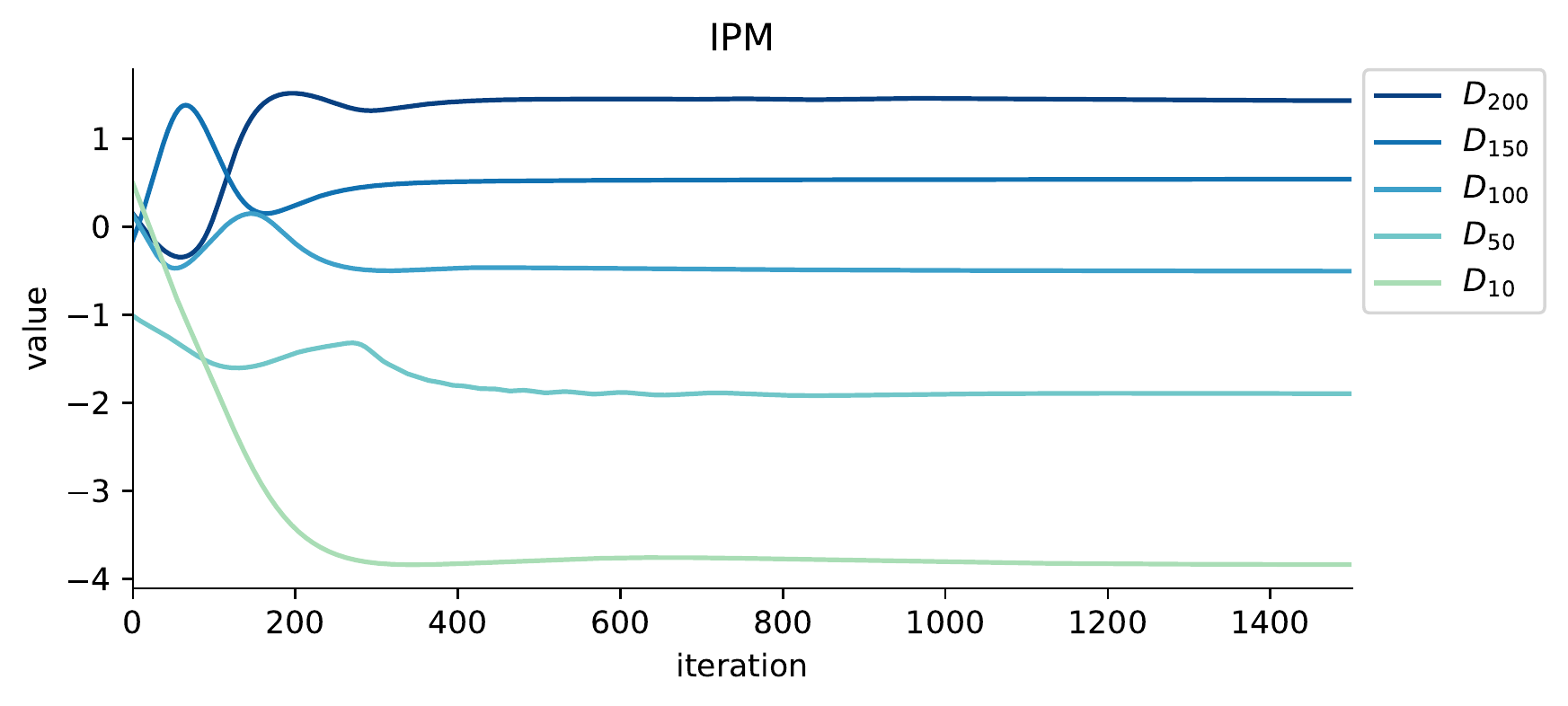}\\
	\caption{Evolution of some diagonal elements for the Toda and IPM flows.
	Both flows tend to order the elements, but IPM does so in a significantly more stable way.
	}
	\label{fig:diag-values}
\end{figure}


\section{Comments on the diagonalization algorithms}\label{sec:comments}

\subsection{Discretization of the fluid and Toda flows}

We first describe the setup for numerical simulations yielding \autoref{fig:todaflow} and \autoref{fig:ipmflow}.
The connection between isospectral flows of matrices and incompressible transport equations for functions on some domain $M$ is established via the following discretization procedure.
The notion, due to Zeitlin~\cite{Ze1991} and based on the explicit quantization of Hoppe~\cite{Ho1989}, gives an approach for Casimir preserving numerical discretization schemes for incompressible 2D hydrodynamics. 
In order for the discretization to yield finite-dimensional operators (matrices), the underlying symplectic manifold $M$ has to be compact.
Thus, instead of $T^*S^1$ we work with $S^2$ (for which efficient discretization algorithms are developed in \cite{MoVi2020,CiViMo2023}).

The corresponding discretization gives us projections $\mathcal T_N\colon C^\infty(S^2)\to \mathfrak{u}(N)$ with right inverses $\mathcal T_N^{-1}$ such that (weakly)
\begin{equation*}
	\mathcal T_N^{-1}([\mathcal T_N(f),\mathcal T_N(f)]) \to \{f,g \} \quad\text{as}\quad N\to \infty .
\end{equation*}
The size $N$ of the matrices should thus be seen as a discretization parameter.

In the case of Toda, we confine to the invariant subspace of purely imaginary symmetric matrices $\mathrm{i}\mathrm{S}(N)\subset\mathfrak{u}(N)$, which we  identify with real symmetric matrices $\mathrm{S}(N)$.
This restriction corresponds to the restriction of functions $\{ L\in C^\infty(S^2)\mid L(\theta,\phi) = L(\theta,-\phi) \}$ where $\theta$ is the polar angle and $\phi$ is the azimuthal angle.
As described above, the finite-dimensional Toda flow~\eqref{eq:toda_double_bracket_form} actually evolves on even the  smaller subspace of tridiagonal symmetric matrices, which corresponds to functions on $S^2$ whose spherical harmonics components $L_{\ell m}$ vanish for $|m|>1$.
Furthermore, zonal functions correspond to diagonal matrices, and, in particular, the height function on the sphere $z = \cos\theta$ corresponds to the diagonal matrix $D$ with equidistant entries from $-1$ (south pole) to $1$ (north pole).
We have thus obtained that the finite-dimensional Toda flow \eqref{eq:toda_double_bracket_form} is a numerical discretization of the corresponding continuous Toda flow~\eqref{eq:mainL} for $M=S^2$ (instead of $M=T^*S^1$).

To obtain a matrix discretization of the IPM equation on $S^2$ we need, in addition to the structures just described, a discrete version of the Laplace operator. 
Based on the representation of $\mathfrak{so}(3)$ in $\mathfrak{u}(N)$, such an operator $\Delta_N\colon \mathfrak{u}(N)\to \mathfrak{u}(N)$ is given by Hoppe and Yau~\cite{HoYa1998}.
An efficient method for solving the discrete Poisson equation $\Delta_N W = P$ in $\mathcal O(N^2)$ operations is derived in~\cite{CiViMo2023}. This provides all the ingredients needed  to obtain a matrix discretization of the IPM equation~\eqref{eq:vort_ipm} for $M=S^2$ and $V(\theta,\phi) = z = \cos\theta$.

We simulate the two systems with initial data given by a randomly generated, tridiagonal symmetric $N\times N$ matrix for  $N=256$.
For time integration we use the isospectral preserving method developed in \cite{MoVi2020c}.
Snapshots of the results, visualized as smooth functions via the Hammer projection of the sphere, are given in \autoref{fig:todaflow} and \autoref{fig:ipmflow}.
In both cases the flows tend towards a zonal state (corresponding to a diagonal matrix). 

\medskip

\subsection{Modification of the QR diagonalization via the IPM flow}

The classical QR algorithm for diagonalization of symmetric matrices, which is one of the main  diagonalization algorithms, can be regarded as the Toda flow at integer time moments, see \cite{Deift1}.
The latter (often restricted to tri-diagonal matrices) is called the Toda algorithm, and it not only diagonalizes matrices, but also orders their eigenvalues. For instance, its stopping time, when the numerically found eigenvalues are sufficiently close to the actual ones, has universal nature and is an important research topic, see \cite{Deift2}.

\begin{figure}
	\includegraphics[width=0.9\textwidth]{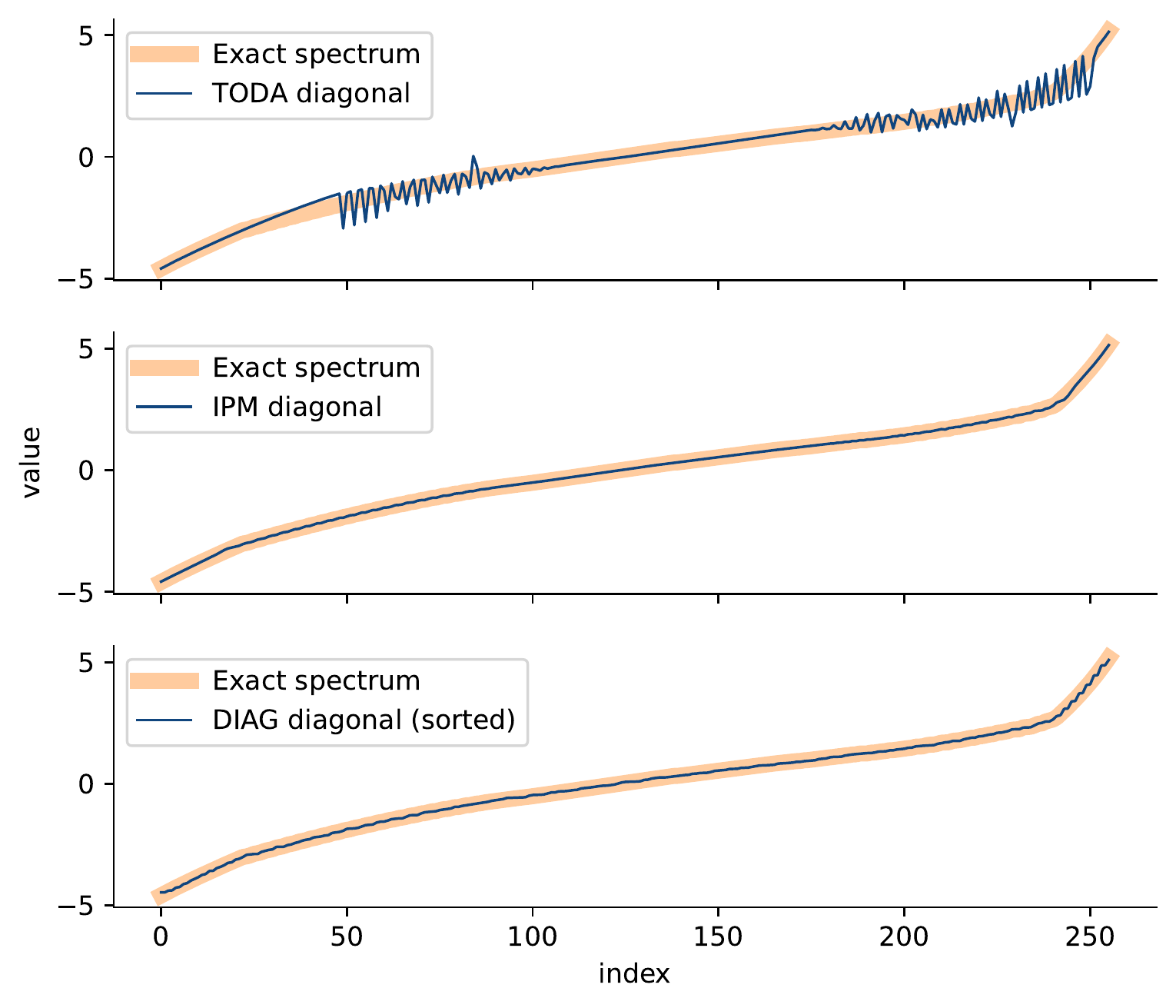}
	\caption{Final diagonal elements compared with exact spectrum for Toda, IPM, and the diagonalizing gradient flow~\eqref{eq:diag_flow}.
	}
	\label{fig:diag-final}
\end{figure}

Our study comparing the Toda and IPM flows suggests the following modification of the  QR diagonalization algorithm for symmetric matrices, which is particularly efficient for matrices of large size. Namely, \autoref{fig:diag-values} compares the speed of convergence to reordered eigenvalues for the $256\times 256$ initial matrix
by our finite-dimensional versions of the Toda flow and the IPM flow. 
The graphs for the 10th, 50th, 100th, 150th, and 200th eigenvalues continue to oscillate for a long time in the Toda flow.
On the other hand, there are no such oscillations for the IPM flow, which quickly arrives at the final values.
The convergence turns out to be so much faster due to the inverse Laplacian as the inertia operator in the IPM equation, which makes the numerical scheme more stable.

This comparison is continued in 
\autoref{fig:diag-final} which describes the final diagonal values for the Toda and IPM algorithms (and also the new diagonalizing flow outlined below) by juxtaposing them with the exact spectrum. 
The result for IPM is almost indistinguishable from the true spectrum, while the Toda flow still exhibits oscillations and unordered segments in the spectrum (after 1500 iterations).
While the detailed study is yet to be done, these figures already suggest that the IPM equation with the Laplace inertia operator might be an efficient replacement for the diagonalization and reordering algorithm: 
due to the presence of $\Delta^{-1}$ in the equation, its flow is stable and converging much faster. 

\medskip

\subsection{An unordered diagonalization algorithm}
The new approach via IPM also suggests new gradient flows for diagonalization of symmetric matrices. 
Indeed, from a numerical point of view, a convergence to a diagonal matrix  (called deflation) is much more essential 
then  ordering  the eigenvalues. 
Thus, it is natural to consider the isospectral gradient flow on symmetric matrices $L$ that aims to maximize the following energy
\begin{equation*}
	F(L) = \frac 12 \sum_{i=1}^N \lvert L_{ii}\rvert^2 .
\end{equation*}
Indeed, recall that the square of the matrix Frobenius norm is ${\rm tr}\,(LL^\top)$. For a symmetric matrix $L=L^\top$
the  norm is constant on its isospectral orbit.  Furthermore, this squared norm is
given by the  sum of squared eigenvalues, so the energy functional $F$ on the orbit is maximized if and only if $L$ is diagonal. 
The differential of this energy at a matrix $L$ is given by the projection (denoted by $D(L)$) onto the diagonal part. Thus the IPM version of the gradient flow for the energy $F$ is given by
\begin{equation}\label{eq:diag_flow}
	\dot L = [\Delta_N^{-1}[D(L),L],L],
\end{equation}
where $\Delta_N$ is the discrete Laplacian described above.
The continuous analog of $D(L)$ for $M=S^2$ is $L^2$-orthogonal projection onto the subspace of zonal functions (i.e., averaging along fixed latitudes).
As can be seen in  \autoref{fig:diag-values-diag-flow}, the convergence to a diagonal matrix is fast, and as stable as in IPM, but contrary to IPM and Toda the elements on the diagonal are not sorted.
As mentioned above, if sorted, the computed spectrum is  precise, similarly to the IPM, cf. the last two diagrams in \autoref{fig:diag-final}.

\begin{figure}
	\centering
	\includegraphics[width=0.95\textwidth]{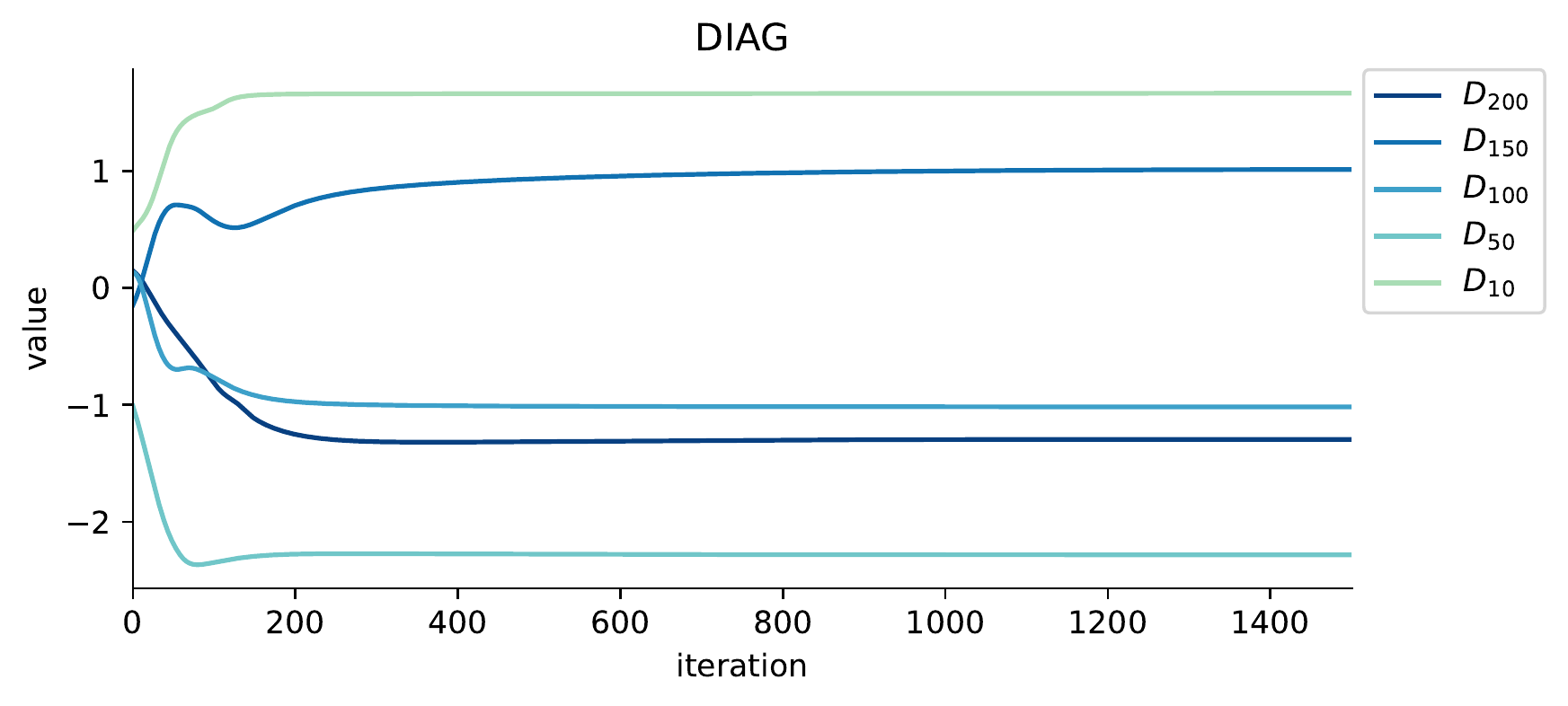}
	\caption{Evolution of some diagonal elements for the diagonalizing gradient flow~\eqref{eq:diag_flow}.
	The convergence is fast, but contrary to Toda and IPM this flow does not strive to sort the elements.
	}
	\label{fig:diag-values-diag-flow}
\end{figure}

Finally, in \autoref{fig:diagflow} we see how the gradient flow for the energy $F$ interpreted on the sphere strives to a zonal state, but with bands that are non-ordered from south to north, while instead correlated to the initial configuration.
These findings indicate that the flow~\eqref{eq:diag_flow} might be useful for diagonalization when sorting of eigenvalues is unimportant.

\begin{remark}
The code for the simulations and animations illustrating the dynamics of the three flows are available here:
\begin{center}
	\href{https://github.com/klasmodin/diagonalizing-flows}{https://github.com/klasmodin/diagonalizing-flows}	
\end{center}
\end{remark}

\begin{figure}
	\centering
	\begin{subfigure}[b]{0.49\textwidth}
		\centering
		\caption{Initial time}
		\includegraphics[width=\textwidth]{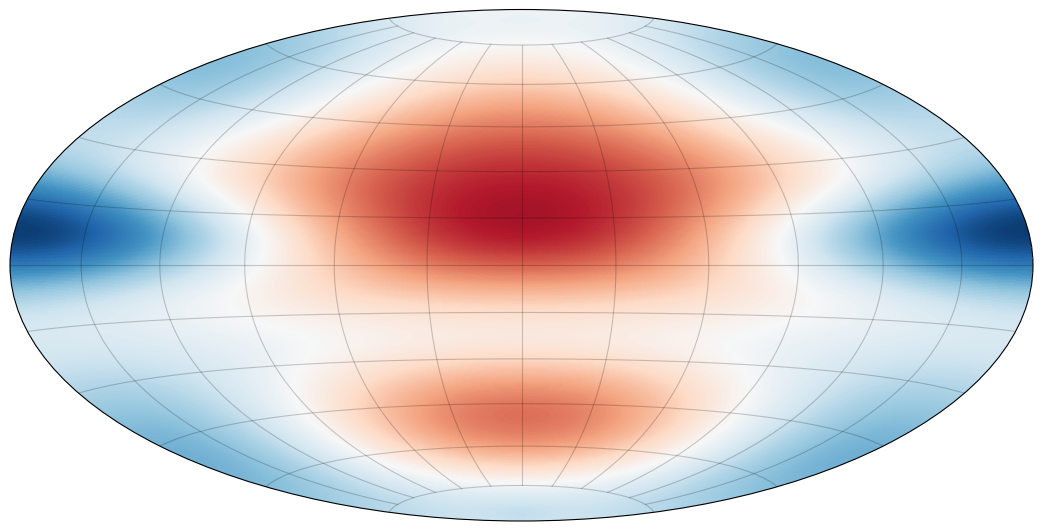}
	\end{subfigure}
	\hfill
	\begin{subfigure}[b]{0.49\textwidth}
		\centering
		\caption{First intermediate time}
		\includegraphics[width=\textwidth]{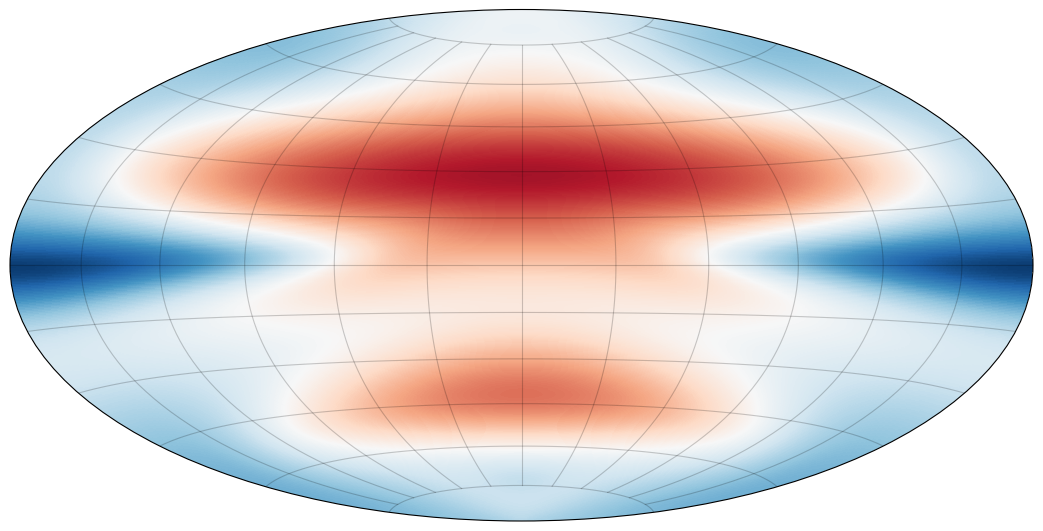}
	\end{subfigure}
	\\
	\begin{subfigure}[b]{0.49\textwidth}
		\centering
		\caption{Second intermediate time}
		\includegraphics[width=\textwidth]{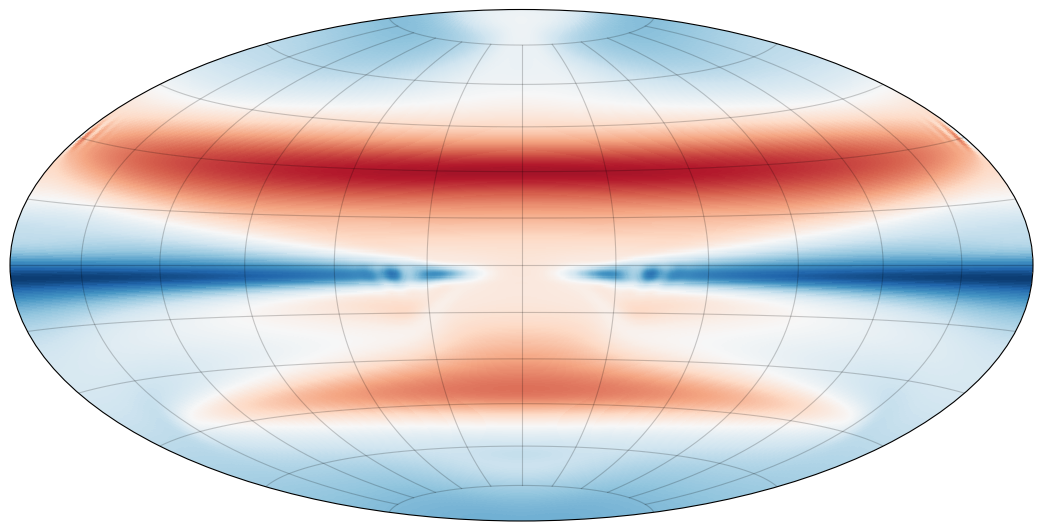}
	\end{subfigure}
	\hfill
	\begin{subfigure}[b]{0.49\textwidth}
		\centering
		\caption{Final time}
		\includegraphics[width=\textwidth]{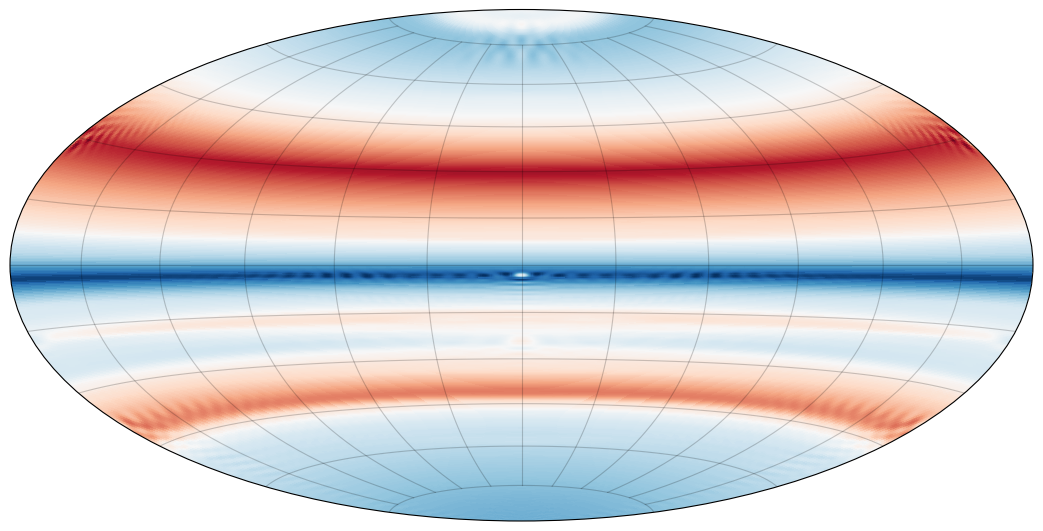}
	\end{subfigure}
	\caption{The diagonalizing gradient flow~\eqref{eq:diag_flow} visualized on the sphere. The flow strives to a zonal configuration, but contrary to Toda and IPM there is no tendency to a monotone zonal state from south to north.}
	\label{fig:diagflow} 
\end{figure}


\section{Appendix A: Universal double bracket flows}\label{sec:universal} 

\subsection{The gradient flow on a Lie group}
Here we consider a framework for Riemannian gradient flows confined to group orbits.
This leads to the proper setting for the double bracket flows in the case of an arbitrary inertia operator.
A similar framework, with focus on shape analysis, is also given in \cite{BaKaMo2022}.

Let $Q$ be a configuration manifold, possibly infinite-dimensional.
Furthermore, let $G$ be a Lie group (or a Fr\'echet--Lie group in the infinite-dimensional case) acting on~$Q$ from the left by a smooth action map $\Phi\colon G\times Q\to Q$; the action of $g\in G$ on $q\in Q$ is denoted $g\cdot q$.
This action is typically neither free, nor transitive.
The orbit of $q\in Q$ is
\begin{equation}
	\Orb{q} = \{ g\cdot q\mid q\in G \}.
\end{equation}
We think of $Q$ as the space of `shapes' and the orbit $\Orb{q}$ represents all possible deformations of~$q$.
For a given \emph{template shape} $q_0$, our objective is to study gradient flows on $\Orb{q_0}$.
Notice, however, that we do not assume that $Q$ is Riemannian; the Riemannian structure on $\Orb{q_0}$ is instead inherited from $G$.
Before we state the main result of this section, we need a few concepts from geometric mechanics (\emph{cf.}~\cite{MaRa1999}).

By differentiating the action map at the identity we obtain the infinitesimal action map $\mathfrak{g}\times Q\to T Q$, where $\mathfrak{g}= T_e G$ is the Lie algebra of $G$.
The infinitesimal action of $v\in\mathfrak{g}$ on $q\in Q$ is denoted $v\cdot q$.
This  linear map in $v$ is the (cotangent bundle) momentum map:
\begin{definition}
	The \emph{momentum map} $J\colon T^*Q\to \mathfrak g^*$ is defined by
	\begin{equation}
		\pair{J(q,p),v} = \pair{p,v\cdot q} 	\qquad \forall\, v\in \mathfrak{g},
	\end{equation}
	where $T^*Q$ denotes the cotangent bundle of~$Q$.\footnote{If $Q$ is an infinite-dimensional Fr\'echet manifold, the cotangent bundle $T^*Q$ is given in terms of the \emph{regular dual} (\emph{cf.}\ \cite{KhWe2009}), defined so that $T_q^*Q\simeq T_qQ$.}
\end{definition}

Next, we introduce a Riemannian structure on $G$.
\begin{definition}
	A Riemannian metric $\inner{\cdot,\cdot}\colon TG\times TG\to \RR$ on $G$ is called \emph{right-invariant} if
	\begin{equation}
		\inner{u,v}_e = \inner{u\cdot g, v\cdot g}_{g}, \qquad \forall\, g\in G,\quad \forall\, u,v\in \mathfrak{g},
	\end{equation}
	where $u\cdot g$ denotes the tangent lifted right action of $g$ on $u$.
\end{definition}

A right-invariant metric is completely determined by the inner product $\inner{\cdot,\cdot}_e$.

\begin{definition}
	Let $\inner{\cdot,\cdot}$ be a right-invariant metric on $G$.
	Then the \emph{inertia operator} $A\colon \mathfrak{g}\to \mathfrak{g}^*$ is defined by
	\begin{equation}
		\pair{Av,u} = \inner{v,u}_e.
	\end{equation}
\end{definition}

Since $G$ acts on $Q$ from the left, the action map induces a Riemannian structure on $\Orb{q}$.
To see this, we first need the notion of horizontal vectors on $G$.

\begin{definition}
	The \emph{vertical distribution} associated with the action of $G$ on $q\in Q$ is the subbundle of $TG$ given by
	\begin{equation}
		\Ver_g = \{ v\cdot g \in T_g G \mid v\cdot g\cdot q = 0 \}.
	\end{equation}
	If $\inner{\cdot,\cdot}$ is a right-invariant metric on $G$, then the corresponding \emph{horizontal distribution} is given by
	\begin{equation}
		\Hor_g = \Ver_g^\bot
	\end{equation}
	where the complement is taken with respect to $\inner{\cdot,\cdot}$.
\end{definition}

\begin{lemma}
	Assume that $\Orb{q}$ is a submanifold of $Q$.
	Then any right-invariant Riemannian metric $\inner{\cdot,\cdot}$ induces a Riemannian metric $\met$ on $\Orb{q}$ fulfilling
	\begin{equation}\label{eq:met_on_orbit}
		\met_{g\cdot q}(v\cdot g\cdot q,v\cdot g\cdot q) = \inner{v,v}_e \qquad \forall\, v\cdot g\in \Hor_g .
	\end{equation}
\end{lemma}

\begin{proof}
	Since $\Orb{q}$ is a manifold, the mapping $\pi\colon G\to Q$ defined by $\pi(g) = g\cdot q$ is a submersion.
	Thus, $T_g \pi\colon \Hor_g \to T_{\pi(g)}\Orb{q}$ is a linear isomorphism.
	Now, for $(x,\dot x) \in T\Orb{q}$, take any $g$ such that $x = \pi(g)$ (which exist by definition of the group orbit).
	Define the metric at $x$ by
	\begin{equation}
		\met_x(\dot x,\dot x) = \inner{\underbrace{(T_g\pi)^{-1}\dot x}_{v\cdot g},\underbrace{(T_g\pi)^{-1}\dot x}_{v\cdot g}}_g.
	\end{equation}
	By right-invariance, this metric satisfies \eqref{eq:met_on_orbit} and is independent of the choice of $g$.
\end{proof}

Let $G$ be a Lie group acting from the left on a manifold $Q$ of shapes.
Let $q_0\in Q$ and let $F\colon Q\to \RR$ be a function on $Q$.
We are interested in finding the minimum of $F$ on the $G$-orbit of $q_0$, that is, we want to minimize the function $E\colon G\to \RR$ defined by
\begin{equation}
	E(g) = F(g\cdot q_0).
\end{equation}

If $G$ is equipped with a right-invariant Riemannian metric $\met$, defined by an inertia operator $A\colon\mathfrak g\to\mathfrak g^*$, one can define the gradient vector field $\nabla E$ on $G$ by
\begin{equation}
	\met_g(\nabla E(g),\dot g) = \pair{\ud E,\dot g}.
\end{equation}
Our aim is to solve the minimization problem by considering the gradient flow
\begin{equation}\label{eq:gradient_flow_on_G}
	\dot g = -\nabla E(g).
\end{equation}

\begin{proposition}\label{prop:gradient_of_E}
	The gradient $\nabla E$ is given by
	\begin{equation}
		\nabla E(g) = \xi\cdot g.
	\end{equation}
	where $\xi\in \mathfrak g$ is given by
	\begin{equation}
		\xi = A^{-1}J(g\cdot q_0, \ud F(g\cdot q_0)).
	\end{equation}
\end{proposition}

\begin{proof}
	By definition of the gradient we have
	\[
	    \inner{\nabla E, \dot g}_{g} = \frac{d}{dt}E(g(t)) = \pair{d F,\xi\cdot (g\cdot q_0)}.
	\]
	where $\xi = \dot g\cdot g^{-1}$.
	Now from the definition of the momentum map it  follows that
	\[
	    \inner{\nabla E, \dot g}_{g} = \pair{J(g\cdot q_0, dF), \xi} = \inner{A^{-1}J(g\cdot q_0, dF), \dot g\cdot g^{-1}}_e
	\]
	The result follows since the metric is right invariant.
\end{proof}

\begin{proposition} \label{prop:descending_flow}
	The gradient flow \eqref{eq:gradient_flow_on_G} induces a gradient flow on the $G$-orbit of $q_0$, given by
	\begin{equation}\label{eq:flow_on_Q}
		\dot q = -u(q)\cdot q
	\end{equation}
	where
	\begin{equation}
		u(q) = A^{-1}J(q, \ud F(q))
	\end{equation}
\end{proposition}

\begin{proof}
	Follows from Prop.~\ref{prop:gradient_of_E}.
\end{proof}

\begin{definition}\label{def:double}
The double-bracket flow on the dual Lie algebra $\mathfrak{g}^*$ with the inertia operator $A$ and a potential function $F$ on  $\mathfrak{g}^*$ 
is defined as follows:
\[
    \dot m = \operatorname{ad}^*_{A^{-1}\operatorname{ad}^*_{\ud F(m)}(m)}(m) .
\]
For the quadratic potential $F(m) = \langle m, A^{-1}m\rangle$ the corresponding flow is
\[
    \dot m = \operatorname{ad}^*_{A^{-1}\operatorname{ad}^*_{A^{-1}m}(m)}(m) .
\]
\end{definition}

\begin{proposition}
The double-bracket flow on   $\mathfrak{g}^*$ is the gradient for the restriction of $F$ on each coadjoint orbit.
\end{proposition}

\begin{proof}
Indeed, consider  the special case where $Q = \mathfrak{g}^*$.
The action is given by $g\cdot m = \operatorname{Ad}_g^* m$.
The momentum map is thereby given by
\[
    \langle J(m,\xi), v \rangle = \langle \xi, \operatorname{ad}^*_v(m)\rangle
    = \langle \operatorname{ad}_v(\xi), m\rangle
    = \langle -\operatorname{ad}_\xi(v), m\rangle
    = \langle -\operatorname{ad}^*_\xi(m), v\rangle
\]
Then from \eqref{eq:flow_on_Q} we obtain the equation(s) in Definition \ref{def:double}, and the statement follows from Proposition \ref{prop:descending_flow}. The corresponding flow tries to minimize the energy on a specific co-adjoint orbit.
\end{proof}

Note also that this flow is always orthogonal (with respect to $A^{-1}$) to the (Hamiltonian) Euler-Arnold flow: one of them is tangent to levels of the Hamiltonian, while the other is orthogonal to the levels of the same function regarded as a potential. One can also see this directly, since
\[
    \langle \operatorname{ad}^*_{A^{-1}\operatorname{ad}^*_{A^{-1}m}(m)}(m), A^{-1}\operatorname{ad}^*_{A^{-1}m}(m)\rangle=
\]
\[     \langle m , [A^{-1}\operatorname{ad}^*_{A^{-1}m}(m),
    A^{-1}\operatorname{ad}^*_{A^{-1}m}(m)]
    \rangle = 0 .
\]



\subsection{Contraction property of the flow}

Consider the gradient flow \eqref{eq:gradient_flow_on_G}.

\begin{proposition}\label{prop:contraction}
	Let $d\colon G\times G\to \RR$ denote the Riemannian distance function induced by the right-invariant metric associated with the gradient flow~\eqref{eq:gradient_flow_on_G}.
	Let $\gamma\colon [0,T)\to G$ be a solution curve. 
	Then, for all $t\in [0,T)$ 
	\begin{equation}
		d(\gamma(0),\gamma(t))^2 \leq t\Big( E\big(\gamma(0)\big) - E\big(\gamma(t)\big) \Big).
	\end{equation}
\end{proposition}

\begin{proof}
	First notice that
	\[
		\frac{d}{dt} E(\gamma(t)) = \inner{\nabla E(\gamma),\dot\gamma}_\gamma = -\inner{\dot\gamma,\dot\gamma}_\gamma 
		= -\inner{\dot\gamma\circ\gamma^{-1}, \dot\gamma\circ\gamma^{-1}}_e 
		= -\inner{v,v}_e 
	\]
	for the vector field $v(t)$ generating the curve $\gamma(t)$. Denote by
	$\norm{v}^2_A := \inner{v,v}_e = \pair{Av,v}  $.
Since any curve between $\gamma(0)$ and $\gamma(t)$ cannot exceed the length of the geodesic between the points we have
	\[
		d(\gamma(0),\gamma(t)) \leq \int_0^t \norm{v(s)}_A \uud s \leq \sqrt{t}\left(\int_0^t \norm{v(s)}_A^2\uud s \right)^{1/2},
	\]
	where the last inequality is Cauchy--Schwartz on $L^2([0,t])$.
\end{proof}

\begin{remark}
\autoref{prop:contraction} in relation to the IPM flow then gives the following result: the ``Arnold fluid-distance'' between $\varphi(0)$ and $\varphi(t)$ is bounded by
\[
d(\varphi(0),\varphi(t))^2 \leq t \int_M V(\rho(0)-\rho(t)) .
\]
\end{remark}


\section{Appendix B: A direct continuous Toda limit}

Start from the Hamiltonian for the finite-dimensional Toda lattice in $(q_j,p_j)$ variables.
Take the limit $n\to\infty$ to obtain a continuous system on $T^*\mathrm{Dens}(\mathbb{R})$.
Under this limit we have
\begin{equation}
	q_{j+1}-q_{j} \to \varphi'(z)
\end{equation}
where $\varphi$ is the diffeomorphism describing how each point on the line has moved.
Since we also have $\varphi'(z) > 0$ and therefore $\rho = \mathrm{Jac}(\varphi^{-1}) = 1/\varphi'\circ\varphi$, we get the potential
\begin{equation}
	U(\varphi) = \int_\mathbb{R} V(\varphi') dz = \int_\mathbb{R} V(1/\rho\circ\varphi^{-1})dz,
\end{equation}
where the potential $V$ is
\begin{equation}
	V(r) =\exp(-2r).
\end{equation}
The gradient of $U$ is computed as
\begin{equation}
	\frac{d}{dt}U(\varphi) = \int_\mathbb{R}V'(\varphi')\dot\varphi' dz =
	\int_\mathbb{R} \left(-\frac{\partial}{\partial x}V'(\varphi')\right) \dot\varphi\; dz
\end{equation}

Newton's equations on the space of diffeomorphisms is thereby
\begin{equation}
	\ddot\varphi = \frac{\partial}{\partial z}V'(\varphi') = -2\frac{\partial}{\partial z} \exp(-2\varphi').
\end{equation}
The fluid formulation of this system is
\begin{equation}
	\dot v + \nabla_v v = \left(\frac{\partial}{\partial z}V'(\varphi') \right)\circ\varphi^{-1},
	\qquad \dot\varphi = v\circ\varphi .
\end{equation}



Let us now make the following change of variables
\begin{equation}\label{eq:notations}
	a = \exp(-\varphi'), \qquad b = \dot\varphi .
\end{equation}
Since the density is uniform, the variable $b$ is interpreted physically as the Lagrangian momentum.
Direct calculations now yield
\begin{equation}
	\dot a = -a \frac{\partial}{\partial z}\dot\varphi = -a \frac{\partial}{\partial z}b
\end{equation}
and, likewise,
\begin{equation}
	\dot b = \ddot\varphi = -2\frac{\partial}{\partial z}a^2 . 
\end{equation}
This is the same equation as in the first continuous limit, which can be rewritten as
\begin{equation}\label{eq:abflow}
	\dot a = -a \frac{\partial}{\partial z}b, \qquad \dot b = -2 \frac{\partial}{\partial z}a^2 .
\end{equation}


\bigskip

{\bf Declaration section.} The authors have no competing interests to declare that are relevant to the content of this article. 
Data sharing is not applicable to this article as no datasets were generated or analysed during the current study.
\medskip


\bibliographystyle{amsplainnat} 
\bibliography{ipm_geometry} 

\end{document}